\documentclass[12pt,a4paper,twoside]{amsart}

\usepackage{amssymb}

\newcommand{\Comp}{{\rm Comp }}
\newcommand{\R}{\operatorname{Re}}
\newcommand{\cinterv}[1]{\mathopen[{#1}\mathclose]}
\newcommand{\cointerv}[1]{\mathopen[{#1}\mathclose)}
\newcommand{\ocinterv}[1]{\mathopen({#1}\mathclose]}
\newcommand{\closure}[1]{\overline{#1}}
\newcommand{\D}{\mathbb{D}}
\newcommand{\T}{\mathbb{T}}
\newcommand{\Hplane}{\mathbb{H}}
\newcommand{\abs}[1]{\lvert#1\rvert}
\newcommand{\bigabs}[1]{\bigl\lvert#1\bigr\rvert}
\newcommand{\biggabs}[1]{\biggl\lvert#1\biggr\rvert}
\newcommand{\boundary}{\partial}
\newcommand{\norm}[1]{\lVert#1\rVert}
\newcommand{\bignorm}[1]{\bigl\lVert#1\bigr\rVert}
\newcommand{\epsil}{\varepsilon}

\numberwithin{equation}{section}

\newtheorem{lemma}{Lemma}[section]

\newtheorem{theorem}[lemma]{Theorem}
\newtheorem*{mainteo}{Main Theorem}
\newtheorem{proposition}[lemma]{Proposition}
\newtheorem{problem}[lemma]{Problem}
\newtheorem{ejemplo}[lemma]{Example}

\theoremstyle{remark}

\newtheorem{remark}[lemma]{Remark}

\begin{document}

\title[On the connected component of compact composition operators]
{On the connected component of compact composition operators on the
Hardy space}

\author[E.\ A.\ Gallardo-Guti\'errez]{Eva A.\ Gallardo-Guti\'errez}
\address{Departamento de Matem\'aticas, Universidad de Zaragoza,
Plaza San Francisco s/n, 50009 Zaragoza, Spain.}
\email{eva@unizar.es}
\author[M.\ J.\ Gonz\'alez]{Mar\'{\i}a J.\ Gonz\'alez}
\address{Departamento de Matem\'aticas, Universidad de C\'adiz,
Apartado 40, 11510 Puerto Real (C\'adiz), Spain.}
\email{majose.gonzalez@uca.es}
\author[P.\ J.\ Nieminen]{Pekka J.\ Nieminen}
\address{Department of Mathematics and Statistics, University of
Helsinki, PO Box 68, FI-00014 Helsinki, Finland.}
\email{pjniemin@cc.helsinki.fi}
\author[E.\ Saksman]{Eero Saksman}
\address{Department of Mathematics and Statistics, University of
Jyv\"askyl\"a, PO Box 35, FI-40014 Jyv\"askyl\"a, Finland.}
\email{saksman@maths.jyu.fi}

\thanks{First author was partially supported by Plan Nacional I+D grant
no.\ MTM2006-06431 and Gobierno de Arag\'on research group
\emph{An\'alisis Matem\'atico y Aplicaciones}, ref.\ DGA E-64.
Second author was partially supported by Plan Nacional I+D grant
no.\ MTM2005-00544 and 2005SGR00774. Third author was supported by
the Finnish Graduate School in Mathematical Analysis and Its
Applications, and the Academy of Finland, project no.\ 118422.
Fourth author was supported by the Academy of Finland, projects
no.\ 113826 and 118765.}

\subjclass[2000]{Primary 47B33. Secondary 30D55, 47B38.}
\keywords{composition operator, Aleksandrov--Clark measures,
compact operator}
\date{15 June 2007.}


\begin{abstract}
We show that there exist non-compact composition operators in
the connected component of the compact ones on the classical Hardy
space $\mathcal{H}^2$.  This answers a question posed by Shapiro and
Sundberg in 1990.  We also establish an improved version of a theorem
of MacCluer, giving a lower bound for the essential norm of a
difference of composition operators in terms of the angular
derivatives of their symbols.  As a main tool we use
Aleksandrov--Clark measures.
\end{abstract}

\maketitle

\section{Introduction}
\label{sec:Intro}

Let $\D$ denote the open unit disc of the complex plane
and $\mathcal{H}^2$ the classical Hardy space, that is, the space
of analytic functions $f$ on $\D$ for which the norm
\[
\norm{f}_2 = \left (\sup_{0\leq r<1} \int_{0}^{2\pi}
\abs{f(re^{i\theta})}^2 \, \frac{d\theta}{2\pi}\right )^{1/2}
\]
is finite. By a variant of Fatou's theorem, any Hardy function $f$
has non-tangential limits on the boundary of the unit disc except on
a set Lebesgue measure zero (see \cite{Du}, for instance).   Moreover,
$\norm{f}_2$ equals the $L^2$-norm of the boundary function.
Throughout
this work, $f(e^{i\theta})$ will denote the non-tangential limit of
$f$ at $e^{i\theta}$.

If $\varphi$ is an analytic map which takes $\D$ into itself, a result
proved by Littlewood in 1925 ensures that the composition operator
induced by $\varphi$,
\[
C_{\varphi} f = f\circ \varphi,
\]
is always a bounded linear operator on $\mathcal{H}^2$. The properties
of such operators on $\mathcal{H}^2$ and many other function spaces
have been studied extensively during the past few decades.  We refer
the reader to the monographs \cite{CMc,ShBook} for an
overview of the field as of the early 1990s.

Starting from Earl Berkson's pioneering work \cite{Be}, many authors
have focused attention on the topological structure of the set
$\Comp(\mathcal{H}^2)$ of all composition operators on
$\mathcal{H}^2$.  Here $\Comp(\mathcal{H}^2)$ is usually
equipped with the metric induced by the operator norm.
A remarkable contribution in this area is due to Joel H.\ Shapiro and
Carl Sundberg~\cite{ShSuIso}, who provided several results and examples
to describe the isolated members of $\Comp(\mathcal{H}^2)$.
Towards the end of their paper, they also raised the general problem of
determining the connected components of $\Comp(\mathcal{H}^2)$, and
suggested the following conjecture:
\begin{itemize}
\item[(A)]
\emph{$C_\varphi$ and $C_\psi$ lie in the same component of
$\Comp(\mathcal{H}^2)$ if and only if $C_\varphi-C_\psi$ is compact.}
\end{itemize}

The most important special case of this conjecture, mentioned
explicitly in~\cite{ShSuIso}, states that the compact composition
operators themselves form a component in $\Comp(\mathcal{H}^2)$.
In fact, Shapiro and Sundberg observed that the collection of the
compact composition operators on $\mathcal{H}^2$ is arcwise connected,
so the remaining question can be stated as follows:
\begin{itemize}
\item[(B)]
\emph{Let $\Comp_K(\mathcal{H}^2)$ be the component of
$\Comp(\mathcal{H}^2)$ that contains all the compact composition
operators.  Does any non-compact composition operator belong to
$\Comp_K(\mathcal{H}^2)$?}
\end{itemize}

The general form (A) of the Shapiro--Sundberg conjecture has recently
been answered negatively by Moorhouse and Toews~\cite{MoTo} and
Bourdon~\cite{Bo}.  They have provided fairly simple and concrete
examples of symbols $\varphi$ and $\psi$ such that the operators
$C_\varphi$ and $C_\psi$ lie in the same component of
$\Comp(\mathcal{H}^2)$ but have a non-compact difference.  However, in
those examples both operators are non-compact, leaving question~(B)
unanswered.

In this work, we will show that the special case of the
Shapiro--Sundberg conjecture fails, too.  That is, we will give an
affirmative answer to question~(B).

\begin{mainteo}
For $0 \leq t \leq 1$ there are analytic maps $\varphi_t\colon
\D \to \D$ such that $t \mapsto C_{\varphi_t}$ is
a continuous map from $[0,1]$ into $\Comp(\mathcal{H}^2)$,
where $C_{\varphi_0}$ is compact and $C_{\varphi_1}$ is
non-compact on $\mathcal{H}^2$.
\end{mainteo}

Let us point out an important result of Barbara
MacCluer~\cite{Mc} which states that if two composition operators
belong to the same component in $\Comp(\mathcal{H}^2)$, then their
symbols must have the same angular derivative (possibly infinity) at
each point of the unit circle $\T = \partial\D$.
Hence any symbol that induces an operator belonging to
$\Comp_K(\mathcal{H}^2)$ cannot have a finite angular derivative at
any point of $\T$.  This indicates that the construction of the
map $\varphi_1$ above is probably not an elementary task.
In particular, since non-existence of finite angular derivatives
characterizes compact composition operators induced by finitely
valent symbols, the valence of $\varphi_1$ has to be infinite.

As a main tool in the proof of Main Theorem we will employ
Aleksandrov--Clark measures.  These measures, associated to any
analytic self-map of the unit disc, have lately found several
applications in the study of composition operators (see
Section~\ref{sec:AC}).  The essence of our argument comprises a
construction of a family of certain continuously singular measures
on $\T$, one for each point of $\cinterv{0,1}$, which are then used to
define the maps $\varphi_t$ in terms of their Aleksandrov--Clark
measures.

The rest of the paper is organized as follows.  In
Section~\ref{sec:AC}, we collect some preliminaries on
Aleksandrov--Clark measures and composition operators.
In Section~\ref{sec:Mac}, we revisit the theorem of MacCluer cited
above and strengthen it slightly.  This result will provide a clue
for part of the proof of our Main Theorem, which is then carried
out in Section~\ref{sec:Main} (see, in particular,
Remark~\ref{re:Heur}).  Finally, in Section~\ref{sec:Further}
we make some additional observations related to Main Theorem
and also present some open questions that arise from our work.

We finally remark that the questions raised by Shapiro and Sundberg
have been studied in many classical function spaces besides the
original $\mathcal{H}^2$.  See, for example,
\cite{McOhZh,AGL,HO,HaMac,Mo,KrMo}.  In most cases the situation seems
to be considerably easier than in the setting of $\mathcal{H}^2$.  In
particular, for the standard Bergman space $\mathcal{A}^2$, MacCluer's
theory shows that the compact composition operators do form a component
of $\Comp(\mathcal{A}^2)$ (see Remark~\ref{re:Bergman}).  Also, in the
setting of $\mathcal{H}^\infty$, the space of bounded analytic
functions, a complete description of the component structure of
$\Comp(\mathcal{H}^\infty)$ was found in~\cite{McOhZh}.

\section{Aleksandrov--Clark measures}
\label{sec:AC}

In this section we collect some preliminaries and background on
Aleksandrov--Clark measures and their relation to composition
operators.  For more information on these measures and their
applications in other areas of analysis, we refer the reader to
the lecture notes \cite{Sa}, the book \cite{CMR}
and the surveys \cite{MaSt,PoSa}.

\subsection{Definition}
\label{sec:ACDef}

Let $\varphi$ be an analytic self-map of $\D$. For any
$\alpha \in \T$, the real part of the function
$(\alpha+\varphi)/(\alpha-\varphi)$ is positive and harmonic in
$\D$, so it may be expressed as the Poisson integral of a
positive Borel measure $\tau_{\varphi,\alpha}$ supported on
$\T$.  That is,
\[
   \R \frac{\alpha+\varphi(z)}{\alpha-\varphi(z)}
   = \frac{1-\abs{\varphi(z)}^2}{\abs{\alpha-\varphi(z)}^2}
   = \int_{\T} P_z \,d\tau_{\varphi,\alpha},
\]
where $P_z(\zeta) = (1-\abs{z}^2)/\abs{\zeta-z}^2$ is the Poisson
kernel for $z \in \D$.  The family of measures
$\{ \tau_{\varphi,\alpha}: \alpha\in\T \}$ are called the
\emph{Aleksandrov--Clark measures} associated to $\varphi$.

For any Borel measure $\tau$ on $\T$, we write
$\tau = \tau^a \,dm + \tau^s$ for the Lebesgue decomposition
of $\tau$, so that $\tau^a$ is the density of the absolutely
continuous part, $m$ is the normalized Lebesgue measure on $\T$
and $\tau^s$ is singular.  It follows from the basic properties of
Poisson integrals that
\[
    \tau_{\varphi,\alpha}^a(\zeta) =
    \frac{1-\abs{\varphi(\zeta)}^2}{\abs{\alpha-\varphi(\zeta)}^2}.
\]
Furthermore, $\tau_{\varphi,\alpha}^s$ is carried by the set where
$\varphi(\zeta) = \alpha$.

\subsection{Angular derivatives}
\label{sec:Angular}

Recall that if the quotient $(\varphi(z)-\eta)/(z-\zeta)$ has a
finite non-tangential limit at $\zeta \in \T$ for some $\eta \in \T$,
then this limit is called the \emph{angular derivative} of $\varphi$
at $\zeta$ and denoted by $\varphi'(\zeta)$.  It satisfies
$\varphi'(\zeta) = \abs{\varphi'(\zeta)} \overline{\zeta}\eta$
where $\eta = \varphi(\zeta)$.  A nice feature of the
Aleksandrov--Clark measures is that their discrete parts (i.e.\ mass
points, or atoms) have a perfect correspondence with the finite
angular derivatives of $\varphi$:

\begin{itemize}
\item
\emph{The map $\varphi$ has a finite angular derivative at
$\zeta \in \T$
if and only if there is $\alpha \in \T$ such that
$\tau_{\varphi,\alpha}(\{\zeta\}) > 0$.  In that case
$\varphi(\zeta) = \alpha$ and $\abs{\varphi'(\zeta)} =
\tau_{\varphi,\alpha}(\{\zeta\})^{-1}$.}
\end{itemize}
For the proof of this result convenient references are \cite{CMR,Sa},
where it is established in conjunction with the classical
Julia--Carath\'{e}odory theorem.

\subsection{Relation to composition operators}
\label{sec:CompOp}

To bring Aleksandrov--Clark measures into the theory of composition
operators, we follow
Sarason's~\cite{Sar} idea of describing composition operators as
integral operators acting on the unit circle.  Let us denote
by $\mathcal{M}$ the space of all complex Borel measures on $\T$
endowed with the total variation norm.  Then, if $\mu \in \mathcal{M}$
is given, the Poisson integral $u(z) = \int_\T P_z \,d\mu$
defines a harmonic function on $\D$.  Consequently the function
$v = u \circ \varphi$ is also harmonic, and it follows easily that
$v$ is the Poisson integral of a unique measure
$\nu \in \mathcal{M}$.  Thus it makes sense to define
$C_\varphi\mu = \nu$. One can show that
$C_\varphi\colon \mathcal{M} \to \mathcal{M}$ is bounded and,
furthermore, that $C_\varphi$ restricts to a bounded operator
$L_p \to L_p$, where $L_p = L_p(\T,m)$ for $1 \leq p \leq \infty$.
Moreover, the restriction of $C_\varphi$ to the analytic Hardy spaces
$\mathcal{H}^p$ (viewed as subspaces of $L^p$) coincides with the
standard definition of $C_\varphi$.

By the definition of the Aleksandrov--Clark measures we see that
$\tau_{\varphi,\alpha } = C_\varphi \delta_\alpha$, where
$\delta_\alpha$ is the $\delta$-Dirac measure at $\alpha$.  In
addition, the correspondence $C_\varphi\mu = \nu$ can be written as
\begin{equation} \label{eq:IntOper}
    \int_\T f\,d\nu =
    \int_\T \biggl( \int_\T f\,d\tau_{\varphi,\alpha} \biggr)
            \,d\mu(\alpha)
\end{equation}
for a suitable class of functions $f$.  Indeed, if $f$ is a Poisson
kernel $P_z$, this follows directly from the definitions.  The case
of continuous $f$ is then obtained by approximating with linear
combinations of Poisson kernels.  Finally one may invoke a further
approximation argument
(e.g.\ a monotone class theorem; cf.\ \cite[Sec.~9.4]{CMR}) to
establish \eqref{eq:IntOper} for all bounded Borel functions $f$ on
$\T$.

In~\cite{Sar} Sarason characterized those composition operators
$C_\varphi$ that are compact on $\mathcal{M}$ and $L^1$ by a condition
which says that $\tau_{\varphi,\alpha}^s = 0$ for all $\alpha \in \T$;
that is, the Aleksandrov--Clark measures of $\varphi$ are required to
be absolutely continuous.  Later Shapiro and Sundberg~\cite{ShSuL1}
observed that Sarason's criterion is equivalent to
Shapiro's~\cite{ShEss} characterization of compact composition
operators on $\mathcal{H}^p$, $1 \leq p < \infty$, involving
the Nevanlinna counting function.  Moreover, Cima and
Matheson~\cite{CiMa} have
shown that the essential norm (i.e.\ distance, in the operator norm,
from the compact operators) of any $C_\varphi$ acting on $\mathcal{H}^2$
equals $\sup_\alpha\norm{\tau_{\varphi,\alpha}^s}^{1/2}$.  In
particular, a necessary condition for the compactness of $C_\varphi$
on all the spaces mentioned is that the symbol $\varphi$ has no finite
angular derivative at any point of $\T$.  This condition, however, is
not sufficient unless $\varphi$ is of finite valence (see e.g.\
\cite{ShBook}).

Aleksandrov--Clark measures have also been used to study differences
and more general linear combinations of composition operators
in \cite{KrMo,NiSa,JESh}.  In particular, a characterization for
compact differences of composition operators on $\mathcal{M}$ and
$L^1$ was found in~\cite{NiSa}.

\section{Extension of MacCluer's Theorem}
\label{sec:Mac}

In 1989 Barbara MacCluer obtained the following result concerning
differences of composition operators on $\mathcal{H}^2$.

\begin{theorem}[MacCluer~\cite{Mc}] \label{thm:Mac}
Assume that $\varphi,\psi\colon \D \to \D$ are analytic maps and
$\varphi$ has a finite angular derivative at $\zeta \in \T$.  Then,
unless $\psi(\zeta) = \varphi(\zeta)$ and
$\psi'(\zeta) = \varphi'(\zeta)$, one has
\[
    \norm{C_\varphi-C_\psi}_e^2 \geq \frac{1}{\abs{\varphi'(\zeta)}},
\]
where $\norm{\ }_e$ denotes the essential norm of an operator
on $\mathcal{H}^2$.
\end{theorem}

The relationship between angular derivatives and the atoms
of the Aleksandrov--Clark measures (see Sec.~\ref{sec:Angular})
allows us to restate Theorem~\ref{thm:Mac} as follows:

\begin{itemize}
\item
\emph{Assume that $\tau_{\varphi,\alpha}(\{\zeta\}) > 0$ for some
$\alpha \in \T$.  Then, unless
$\tau_{\psi,\alpha}(\{\zeta\}) = \tau_{\varphi,\alpha}(\{\zeta\})$,
one has
$\norm{C_\varphi-C_\psi}_e^2 \geq \tau_{\varphi,\alpha}(\{\zeta\})$.}
\end{itemize}

Theorem~\ref{thm:Mac} implies that, for each $\zeta \in \T$ and
$d \neq 0$, the set of all $C_\varphi$ with $\varphi'(\zeta) = d$
is both open and closed in $\Comp(\mathcal{H}^2)$ (even in the
topology induced by the essential norm).  Hence a necessary
condition for two composition operators to lie in the same component
(or essential component) of $\Comp(\mathcal{H}^2)$ is that the angular
derivatives of their symbols coincide. In particular, it follows that
if $C_\varphi$ belongs to $\Comp_K(\mathcal{H}^2)$, the component
containing all compact composition operators, then
$\varphi$ has no finite angular derivative at any point of $\T$---
or, equivalently, the Aleksandrov--Clark measure
$\tau_{\varphi,\alpha}$ has no atoms for any $\alpha \in \T$.

\begin{remark} \label{re:Bergman}
MacCluer's work was actually carried out in a general context of
weighted Dirichlet (or Bergman) spaces $\mathcal{D}_\beta$,
$\beta \geq 1$, which includes as special cases the Hardy space
$\mathcal{H}^2$ ($\beta = 1$) as well as the standard Bergman
space $\mathcal{A}^2$ ($\beta = 2$).  For $\beta > 1$ it is known
that the non-existence of finite angular derivatives
is both necessary and sufficient for the compactness of a composition
operator on $\mathcal{D}_\beta$ (see \cite{McSh} or \cite{CMc}).  So,
in these spaces, MacCluer's theorem implies
(e.g.\ by the argument at the beginning of
the preceding paragraph) that the compacts indeed
form a connected component of $\Comp(\mathcal{D}_\beta)$.

In another direction, Kriete and Moorhouse~\cite{KrMo} have recently
obtained various interesting refinements of MacCluer's results.  In
particular, they establish a version of Theorem~\ref{thm:Mac} for
higher-order boundary data of the symbols.
\end{remark}

In this section we will provide a slight improvement of
Theorem~\ref{thm:Mac}.  Our lower bound will involve the whole discrete
part of the Aleksandrov--Clark measure at $\alpha$.  This result
yields some heuristics for our construction in the proof of our Main
Theorem in Section~\ref{sec:Main} (see Remark~\ref{re:Heur}).

\begin{theorem} \label{thm:MacExt}
Let $\varphi,\psi\colon \D \to \D$ be analytic maps and
$\alpha \in \T$.  Write
\[
    Z = \bigl\{ \zeta \in \T :
        0 < \tau_{\varphi,\alpha}(\{\zeta\}) \neq
            \tau_{\psi,\alpha}(\{\zeta\}) \bigr\}.
\]
Then
\[
    \norm{C_\varphi-C_\psi}_e^2 \geq
       \tau_{\varphi,\alpha}(Z).
\]
\end{theorem}

In the proof of Theorem~\ref{thm:MacExt} we will use as test
functions the normalized reproducing kernels
\[
    f_w(z) = \frac{\sqrt{1-\abs{w}^2}}{1-\overline{w}z}.
\]
They have
the property that $\norm{f_w}_2 = 1$ for all $w \in \D$ and
$f_w \to 0$ weakly as $\abs{w} \to 1$, whence
$\norm{C_\varphi-C_\psi}_e \geq \limsup_{\abs{w}\to 1}
\norm{(C_\varphi-C_\psi)f_w}_2$.  We will borrow MacCluer's idea of
letting $w$ approach $\alpha$ along a curve which makes almost right
angle with the radius to $\alpha$.   However, instead of considering
the adjoints of $C_\varphi$ and $C_\psi$ as in \cite{Mc} and
\cite{KrMo}, we will deal with the composition operators themselves.
The following lemma contains the estimates crucial for our argument.

\begin{lemma} \label{le:Kernel}
Let $\varphi\colon \D \to \D$ be analytic and fix $a > 0$.  For
$\delta, \kappa, \lambda, r > 0$, write
\[
    I(\delta,\kappa,\lambda, r)
       = \frac{1}{2\pi}
         \int_{\kappa ra-\lambda r}^{\kappa ra+\lambda r}
         \bigabs{ C_\varphi f_{(1-r)e^{i\kappa r}}
                  ((1-\delta r)e^{it}) }^2 \,dt.
\]
\begin{enumerate}
\item
If $\tau_{\varphi,1}(\{1\}) = a$, then
\[
    \lim_{r\to 0} I(\delta,\kappa,\lambda, r)
       = \frac{a \cdot c(\delta,\lambda)}{1+\delta/a},
\]
where $0 < c(\delta,\lambda) < 1$ and
$\lim_{\lambda\to\infty} c(\delta,\lambda) = 1$ for all
$\delta > 0$.
\item
If $\tau_{\varphi,1}(\{1\}) \neq a$, then
\[
    \lim_{r\to 0} I(\delta,\kappa,\lambda, r)
       = \epsil(\delta,\kappa,\lambda),
\]
where $\lim_{\kappa\to\infty} \epsil(\delta,\kappa,\lambda) = 0$ for
all $\delta,\lambda > 0$.
\end{enumerate}
\end{lemma}

\begin{proof}
Let us fix $\delta,\kappa,\lambda > 0$, and write
$w_r = (1-r)e^{i\kappa r}$ and
$z_r(t) = (1-\delta r)e^{it}$.  Then
\begin{equation} \label{eq:Iintegral}
    I(\delta,\kappa,\lambda,r)
    = \frac{2r-r^2}{2\pi}
      \int_{\kappa ra-\lambda r}^{\kappa ra+\lambda r}
      \frac{dt}{\abs{1-\overline{w_r} \varphi(z_r(t))}^2}.
\end{equation}

We first consider the case when
$\tau_{\varphi,1}(\{1\}) = b$ for some $b > 0$.  That is,
$\varphi(1) = 1$ and $\varphi$ has a finite angular derivative
equal to $1/b$ at $1$. Note that the points $z_r(t)$ involved in
\eqref{eq:Iintegral} for $0 < r < 1$ all lie in a non-tangential
approach region for the point $1$ (whose opening angle depends
on $\delta$, $\kappa$, $a$,  and $\lambda$).  Therefore, for these
$z_r(t)$ we have
\[
    1 - \varphi(z_r(t)) = b^{-1} (1-z_r(t)) + r \epsil_r(t),
\]
uniformly in $t$.  Here and elsewhere in this proof we use
$\epsil_r$ (with or without additional parameters) to denote a
quantity which tends to zero as $r \to 0$.  With this notation,
we may also write $1 - \overline{w_r} = r + i\kappa r + r\epsil_r$
and $1 - z_r(t) = \delta r - it + r\epsil_r(t)$.  Consequently,
\[ \begin{split}
    1 - \overline{w_r}\varphi(z_r(t))
    &= (1 - \overline{w_r}) + \{1 - \varphi(z_r(t))\} + r\epsil_r(t) \\
    &= r(1+\delta/b) + i(\kappa r-t/b) + r\epsil_r(t).
\end{split} \]
We substitute this expression into the integrand in
\eqref{eq:Iintegral} and perform the change of variables
$u = t/ra-\kappa$ to get
\[ \begin{split}
    I(\delta,\kappa,\lambda,r)
    &= \frac{(2r-r^2)ra}{2\pi} \int_{-\lambda/a}^{+\lambda/a}
       \frac{du}{ \abs{r(1+\delta/b) +
          i (\kappa r - \kappa ra/b - rau/b) + r\epsil_r(u)}^2 } \\
    &= \frac{(2-r)a}{2\pi} \int_{-\lambda/a}^{+\lambda/a}
       \frac{du}{ \abs{(1+\delta/b) +
          i ((1-a/b)\kappa - au/b) + \epsil_r(u)}^2 }.
\end{split} \]
Hence
\begin{equation} \label{eq:Ilimit}
    \lim_{r\to 0} I(\delta,\kappa,\lambda,r)
    = \frac{a}{\pi} \int_{-\lambda/a}^{+\lambda/a}
      \frac{du}{ (1+\delta/b)^2 + ((1-a/b)\kappa-au/b)^2 }.
\end{equation}
If $b = a$, this limit equals
\[
    \frac{a}{\pi} \int_{-\lambda/a}^{+\lambda/a}
    \frac{du}{ (1+\delta/a)^2 + u^2 },
\]
which is of the desired form $ac(\delta,\lambda)/(1+\delta/a)$.
On the other hand, if $b \neq a$, then the integrand in
\eqref{eq:Ilimit} tends to zero as $\kappa \to \infty$, uniformly
in $u$.  So, in this case \eqref{eq:Ilimit} goes to zero as
$\kappa \to \infty$.

Finally assume that $\tau_{\varphi,1}(\{1\}) = 0$, so
$\varphi$ has no finite angular derivative at $1$ or
$\varphi(1) \neq 1$.  By the Julia--Carath\'{e}odory theorem,
we now have $(1-\varphi(z))/(1-z) \to \infty$ as $z \to 1$
non-tangentially.
By considerations similar to those in the first part of the proof,
this implies that
$\{1-\overline{w_r}\varphi(z_r(t))\}/r \to \infty$ as $r \to 0$,
uniformly in $t$, and hence
$I(\delta,\kappa,\lambda,r) \to 0$ as $r \to 0$.  We leave the
details to the reader.
\end{proof}

\begin{proof}[Proof of Theorem~\ref{thm:MacExt}]
Without loss of generality, we may take $\alpha = 1$.  We first treat
the case of a single
mass point and then indicate the general argument.  Let us assume
that $\tau_{\varphi,1}(\{1\}) = a \neq \tau_{\psi,1}(\{1\})$ for some
$a > 0$.  Then, for $\delta,\kappa,\lambda > 0$ and small enough
$r > 0$, we have
\[ \begin{split}
    \bignorm{ (C_\varphi-C_\psi) f_{(1-r)e^{i\kappa r}} }_2
       &\geq \biggl( \frac{1}{2\pi}
             \int_{\kappa ra-\lambda r}^{\kappa ra+\lambda r}
             \bigabs{ (C_\varphi-C_\psi) f_{(1-r)e^{i\kappa r}}
                   ((1-\delta r) e^{it}) }^2
             \,dt \biggr)^{1/2}  \\
       &\geq I_\varphi(\delta,\kappa,\lambda, r)^{1/2} -
             I_\psi(\delta,\kappa,\lambda, r)^{1/2},
\end{split} \]
where $I_\varphi$ and $I_\psi$ refer to the integrals of
Lemma~\ref{le:Kernel} corresponding to $\varphi$ and $\psi$, respectively.
Passing to the limit as $r \to 0$, we then get the following type of
lower bound for the essential norm of $C_\varphi-C_\psi$:
\[
    \norm{C_\varphi-C_\psi}_e
       \geq \biggl(
            \frac{a\cdot c(\delta,\lambda)}{1+\delta/a}\biggr)^{1/2}
            - \epsil(\delta,\kappa,\lambda)^{1/2}.
\]
Letting $\kappa \to \infty$, $\lambda \to \infty$ and $\delta \to 0$
now yields $\norm{C_\varphi-C_\psi}_e \geq a^{1/2}$ as desired.

To prove the theorem in full (assuming still $\alpha = 1$), we observe
that the above reasoning is local in the sense that the interval
$\cinterv{\kappa ra-\lambda r,\;\kappa ra+\lambda r}$ shrinks to
$0$ as $r \to 0$.  Let $Z_0 = \{\zeta_1,\ldots,\zeta_n\}$ be any
finite subset of the (possibly infinite) set $Z$, where
$\zeta_k \neq \zeta_l$ for $k \neq l$.  Write $t_k = \arg\zeta_k$ and
$a_k = \tau_{\varphi,1}(\{\zeta_k\})$.  We proceed as above, just
integrating over the union of the intervals
$\cinterv{t_k+\kappa ra_k-\lambda r,\;t_k+\kappa ra_k+\lambda r}$,
$k = 1,\ldots,n$.  Since these are disjoint for small $r$, we get,
after passing to the appropriate limits as above,
\[
    \norm{C_\varphi-C_\psi}_e
       \geq \biggl(\sum_{k=1}^n
       \tau_{\varphi,1}(\{\zeta_k\}) \biggr)^{1/2}
       = \tau_{\varphi,1}(Z_0)^{1/2}.
\]
Finally, if $Z$ is infinite, we take the supremum over all finite
subsets $Z_0 \subset Z$ to complete the proof of the theorem.
\end{proof}

\section{Proof of Main Theorem:\ non-compact composition operators
in the component of compacts}
\label{sec:Main}

In this section we will establish our Main Theorem, giving a positive
answer to the question~(B) stated in Section~\ref{sec:Intro}.
We will actually find a continuous path that connects compact
composition operators to a non-compact one.  Moreover, the same
construction turns out to work for a variety of spaces
in addition to $\mathcal{H}^2$.

\begin{mainteo}
For $0 \leq t \leq 1$ there are analytic maps
$\varphi_t\colon \D \to \D$ such that $C_{\varphi_0}$ is compact and
$C_{\varphi_1}$ is non-compact on $X$, and $t \mapsto C_{\varphi_t}$
is continuous from $[0,1]$ into $\Comp(X)$, where $X$ is any of the
spaces $\mathcal{M}$, $L^p$ or $\mathcal{H}^p$ with
$1 \leq p < \infty$.
\end{mainteo}

We begin with some preliminary observations and lemmas.  First of all,
it is enough to deal with the case $X = \mathcal{M}$.  Indeed, as we
pointed out in Section~\ref{sec:CompOp}, the compactness of
composition operators is equivalent in any two of the spaces
mentioned.  Furthermore, we may apply
interpolation between $L^1$ (a subspace of $\mathcal{M}$) and
$L^\infty$ to conclude that for any $1 \leq p < \infty$ and
$s,t \in \cinterv{0,1}$,
\[ \begin{split}
    &\norm{C_{\varphi_s}-C_{\varphi_t}\colon L^p\to L^p}  \\
    &\qquad
    \leq \norm{C_{\varphi_s}-C_{\varphi_t}\colon L^1\to L^1}^{1/p} \,
       \norm{C_{\varphi_s}-C_{\varphi_t}\colon
             L^\infty\to L^\infty}^{1-1/p}
    \\
    &\qquad
    \leq 2^{1-1/p} \, \norm{C_{\varphi_s}-C_{\varphi_t}\colon
       \mathcal{M}\to\mathcal{M}}^{1/p}.
\end{split} \]
(See e.g.\ \cite[Sec.~4.1]{BeSh} for the classical Riesz--Thorin
interpolation theorem.)

Throughout the proof we will utilize Sarason's way of viewing
composition operators as acting on the unit circle (cf.\
Sec.~\ref{sec:CompOp}).  If $\varphi$ is an analytic self-map of
$\D$ and $E \subset \T$ is a Borel set, we let $\chi_E C_\varphi$
denote the restriction of $C_\varphi$ to $E$.  More precisely, if
$\mu \in \mathcal{M}$ and $C_\varphi\mu = \nu$, then
$\chi_E C_\varphi\mu$ refers to the Borel measure
$B \mapsto \nu(E \cap B)$ on $\T$.  For functions $f \in L^1$, this
simply means that $\chi_E C_\varphi f(\zeta) =
\chi_E(\zeta) f(\varphi(\zeta))$ for $m$-a.e.\ $\zeta \in \T$.
In this context, equation \eqref{eq:IntOper} easily yields that
\begin{equation} \label{eq:MNorm}
    \norm{\chi_E C_\varphi\colon \mathcal{M} \to \mathcal{M}} =
    \sup \{\tau_{\varphi,\alpha}(E) : \alpha \in \T \}.
\end{equation}

We will also need a tool to estimate the size of the difference of
two composition operators in terms of the boundary values
of their symbols.  We use $\rho$ to denote the hyperbolic distance
in $\D$; it is the conformally invariant metric induced by the
arc length element $2\abs{dz}/(1-\abs{z}^2)$ (see e.g.\
\cite[Sec.~I.1]{Ga}).  When working with
hyperbolic distances, it is often convenient to shift to the
right half-plane $\Hplane = \{z': \R z' > 0\}$, where the hyperbolic
metric $\rho_\Hplane$ is induced by the arc length element
$\abs{dz'}/\R z'$.
For any $\alpha \in \T$, this is accomplished through the M\"obius
transformation $z' = (\alpha+z)/(\alpha-z)$, which takes $\D$ onto
$\Hplane$ isometrically relative to $\rho$ and $\rho_\Hplane$.
This transformation we have already encountered in the definition
of Aleksandrov--Clark measures.

\begin{lemma} \label{le:MDiff}
Let $\varphi,\psi\colon \D \to \D$ be analytic, and let $E \subset \T$
be a Borel set such that $\tau_{\varphi,\alpha}(\partial E) =
\tau_{\psi,\alpha}(\partial E) = 0$ for all $\alpha \in \T$.
Also let $0 < \epsil < 1$.
Suppose that for $m$-a.e.\ $\zeta \in E$ the following holds:\
if one of $\varphi(\zeta)$ and $\psi(\zeta)$ is unimodular, then
$\varphi(\zeta) = \psi(\zeta)$, and otherwise
$\rho(\varphi(\zeta),\psi(\zeta)) \leq \epsil$.
Then
\[
    \norm{\chi_E (C_\varphi-C_\psi)\colon
       \mathcal{M} \to \mathcal{M}}
       \leq C \epsil / (1-\abs{\varphi(0)}),
\]
where $C > 0$ is a universal constant.
\end{lemma}

\begin{proof}
We first note that the Poisson kernel functions $P_z$ satisfy the
following estimate:\ for all $z,w \in \D$ with $\rho(z,w) \leq 1$ and
$\alpha \in \T$,
\begin{equation} \label{eq:PEst}
    \abs{P_z(\alpha)-P_w(\alpha)} \leq C\rho(z,w)\, P_z(\alpha),
\end{equation}
where $C > 0$ is a universal constant.  In fact, one may use
the transformation $z' = (\alpha+z)/(\alpha-z)$ to pass to the right
half-plane where \eqref{eq:PEst} becomes
\[
    \abs{\R(z'-w')} \leq C \rho_\Hplane(z',w') \,\R z',
\]
which is easy to verify by geometric reasoning.

Now fix $\alpha \in \T$ and $0 < r < 1$.  Since
$\rho(r\varphi(\zeta),r\psi(\zeta)) \leq \epsil$ for
$m$-a.e.\ $\zeta \in E$, we get by \eqref{eq:PEst} that
\[
    \int_E \biggabs{
       \frac{1-\abs{r\varphi}^2}{\abs{\alpha-r\varphi}^2} -
       \frac{1-\abs{r\psi}^2}{\abs{\alpha-r\psi}^2} } \,dm
    \leq C\epsil \int_E
       \frac{1-\abs{r\varphi}^2}{\abs{\alpha-r\varphi}^2} \,dm
    \leq C\epsil
       \frac{1-\abs{r\varphi(0)}^2}{\abs{\alpha-r\varphi(0)}^2}.
\]
The last inequality was obtained by extending the integral over
the whole circle $\T$ and using the harmonicity of the integrand.
The definition of the Aleksandrov--Clark measures implies that
the absolutely continuous measure
$(1-\abs{r\varphi}^2)/\abs{\alpha-r\varphi}^2\,dm$ converges
to $\tau_{\varphi,\alpha}$ weak* as $r \to 1$.  Similarly
$(1-\abs{r\psi}^2)/\abs{\alpha-r\psi}^2\,dm$ converges to
$\tau_{\psi,\alpha}$.  Therefore, the preceding chain of
inequalities yields, as $r \to 1$,
\[
    \abs{\tau_{\varphi,\alpha}-\tau_{\psi,\alpha}}(E)
       \leq C\epsil
       \frac{1-\abs{\varphi(0)}^2}{\abs{\alpha-\varphi(0)}^2}.
\]
(Here we needed the assumption that $\tau_{\varphi,\alpha}$ and
$\tau_{\psi,\alpha}$ both assign measure zero to the boundary of $E$.)
Hence
\[ \begin{split}
    \norm{\chi_E (C_\varphi-C_\psi)\colon
       \mathcal{M} \to \mathcal{M}}
       &=    \sup \{\abs{\tau_{\varphi,\alpha}-\tau_{\psi,\alpha}}(E) :
             \alpha \in \T \}  \\
       &\leq \frac{2C \epsil}{1-\abs{\varphi(0)}},
\end{split} \]
and the proof is complete.
\end{proof}

We are now in a position to define the maps $\varphi_t$.
Recall from Section~\ref{sec:CompOp} that a composition
operator $C_\varphi$ is non-compact on any of the spaces
mentioned in Main Theorem if and only if at least one of the
Aleksandrov--Clark measures $\tau_{\varphi,\alpha}$ fails to
be absolutely continuous.  On the other
hand, if $C_\varphi$ is required to belong to the component of compact
composition operators, MacCluer's theorem implies that none of
$\tau_{\varphi,\alpha}$ may have atoms.  That is why we have to
consider Aleksandrov--Clark measures with continuous singularity.

Let $\lambda$ be any nontrivial, positive and finite continuously
singular Borel measure on the unit circle $\T$.  For
$0 \leq t \leq 1$, let
\begin{equation} \label{eq:tau1def}
    \tau_{t,1} = m + \chi_{I(0,t)} \lambda,
\end{equation}
where $I(0,t) \subset \T$ is the closed arc connecting the point
$1$ to $e^{2\pi it}$ in the positive direction and, as before,
$m$ denotes the normalized Lebesgue measure.  We consider
the Herglotz integral of $\tau_{t,1}$,
\[
    H\tau_{t,1}(z) =
       \int_\T \frac{\zeta+z}{\zeta-z} \,d\tau_{t,1}(\zeta),
\]
and define the map $\varphi_t$ by
\begin{equation} \label{eq:phi_t}
    \frac{1+\varphi_t}{1-\varphi_t} = H\tau_{t,1},
    \qquad\text{that is,}\qquad
    \varphi_t = \frac{H\tau_{t,1}-1}{H\tau_{t,1}+1}.
\end{equation}
Since the real part of $H\tau_{t,1}$ is the Poisson integral of
$\tau_{t,1}$, we see that $\tau_{t,1}$ becomes the
Aleksandrov--Clark measure of $\varphi_t$ at $1$.  Moreover,
since this Poisson integral is $\geq 1$ everywhere on $\D$,
it follows that $\varphi_t$ either takes $\D$ into the open
disc $\{w : \abs{w-\tfrac{1}{2}} < \tfrac{1}{2}\}$ or is
constant $0$ (for small $t$).  In general, we let
$\tau_{t,\alpha}$ denote the Aleksandrov--Clark measure of
$\varphi_t$ at $\alpha \in \T$.

The compactness statements of Main Theorem are now immediate.
Since $\tau_{1,1} = m + \lambda$ is not absolutely continuous,
the operator $C_{\varphi_1}$ is non-compact.  On the other hand,
$\varphi_0 \equiv 0$, so $C_{\varphi_0}$ is clearly compact.

The hard part of the proof consists of showing that the map
$t \mapsto C_{\varphi_t}$ is indeed continuous.  This will be based
on the following two lemmas.

\begin{lemma} \label{le:tauMeasure}
Let $\epsil > 0$.  There exists $\delta > 0$ such that if
$I \subset \T$ is an arc with $m(I) \leq \delta$, then the
Aleksandrov--Clark measures of the maps $\varphi_t$ satisfy
$\tau_{t,\alpha}(I) \leq \epsil$ for all $t \in [0,1]$ and
$\alpha \in \T$.  In particular, none of $\tau_{t,\alpha}$
have atoms.
\end{lemma}

\begin{proof}
First of all we note that all the measures $\tau_{t,\alpha}$ are
indeed continuous, i.e.\ have no atoms.  For $\alpha = 1$ this is
clear from \eqref{eq:tau1def}.  For $\alpha \neq 1$ we need to note
that since the image of $\varphi_t$ does not touch $\alpha$, the
harmonic function
\begin{equation} \label{eq:ACt}
    \R \frac{\alpha+\varphi_t(z)}{\alpha-\varphi_t(z)}
       = \int_\T P_z \,d\tau_{t,\alpha},
\end{equation}
is bounded and hence $\tau_{t,\alpha}$ is absolutely continuous.

Using \eqref{eq:tau1def} and \eqref{eq:phi_t} one can easily show
that the left-hand side of \eqref{eq:ACt} is continuous as a function
of the pair $(t,\alpha)$ in $\cinterv{0,1}\times\T$.  Since linear
combinations of Poisson kernels
are dense among the continuous functions on $\T$, it follows that
the map $(t,\alpha) \mapsto \tau_{t,\alpha}$ is continuous in the
weak* sense.

Now assume that the claim of the lemma fails.  Then there are arcs
$I_n \subset \T$ and points $t_n \in \cinterv{0,1}$ and
$\alpha_n \in \T$ such that
$\tau_{t_n,\alpha_n}(I_n) > \epsil$ for all $n \geq 1$ while
$m(I_n) \to 0$.  By passing to a subsequence we may assume that the
intervals $I_n$ (i.e.\ their endpoints) converge to a point
$\zeta_0 \in \T$ and also that $t_n \to t_0$ and
$\alpha_n \to \alpha_0$.  Now for each $\eta > 0$ we have
$\tau_{t_n,\alpha_n}(I(e^{-i\eta}\zeta_0,e^{i\eta}\zeta_0))
> \epsil$ whenever $n$ is large enough.  Since the map
$(t,\alpha) \mapsto \tau_{t,\alpha}$ is weak* continuous, it
follows that
$\tau_{t_0,\alpha_0}(I(e^{-i\eta}\zeta_0,e^{i\eta}\zeta_0)) \geq
\epsil$ for all $\eta > 0$, and hence
$\tau_{t_0,\alpha_0}(\{\zeta_0\}) \geq \epsil$.  This is a
contradiction since we observed that $\tau_{t_0,\alpha_0}$ cannot
have atoms.
\end{proof}

\begin{lemma} \label{le:rho}
Fix $t_0 \in \cinterv{0,1}$ and let $I_0 \subset \T$ be an arc whose
midpoint is $e^{2\pi it_0}$.  If $\epsil > 0$ is given, there
exists $\delta > 0$ such that
\[
    \rho(\varphi_{t_0}(\zeta),\varphi_t(\zeta)) \leq \epsil
    \quad\text{for $\zeta \in \T \setminus I_0$}
\]
whenever $\abs{t_0-t} \leq \delta$.
\end{lemma}

\begin{proof}
Assume that $\abs{t_0-t}$ is so small that the distance of the point
$e^{2\pi it}$ to the set $\T \setminus I_0$ is greater than a
positive constant $c$.  Then $H\tau_{t,1} = H\tau_{t_0,1}
\pm H(\chi_{J_t}\lambda)$, where $J_t \subset \T$ is the arc
connecting $e^{2\pi it_0}$ to $e^{2\pi it}$.  Moreover,
for $\zeta \in \T \setminus I_0$ we have
\[
    \abs{H(\chi_{J_t}\lambda)(\zeta)}
       = \biggabs{\int_{J_t}\frac{\xi+\zeta}{\xi-\zeta}\,d\lambda(\xi)}
       \leq \frac{2}{c} \lambda(J_t).
\]
Since this upper bound tends to zero as $t \to t_0$ and
$\R H\tau_{t_0,1} \geq 1$, we see that the distance between
$H\tau_{t,1}(\zeta)$ and $H\tau_{t_0,1}(\zeta)$ in the hyperbolic
metric of the right half-plane tends to zero as $t \to t_0$, uniformly
for $\zeta \in \T \setminus I_0$.  In view of \eqref{eq:phi_t} and the
conformal invariance of the hyperbolic metric, the same conclusion
holds true for the distance of $\varphi_t(\zeta)$ and
$\varphi_{t_0}(\zeta)$ in the metric $\rho$.
\end{proof}

We are now ready to prove the continuity of the map
$t \mapsto C_{\varphi_t}$ with respect to the operator norm
on $\mathcal{M}$.  Let $0 < \epsil < 1$.  By Lemma~\ref{le:tauMeasure}
we can find $\delta > 0$ such that $\tau_{t,\alpha}(I) \leq
\epsil$ for all $t \in [0,1]$ and $\alpha \in \T$ whenever
$I \subset \T$ is an arc with $m(I) \leq \delta$.
For all such $I$, equation~\eqref{eq:MNorm} yields the estimate
\begin{equation} \label{eq:Est1}
    \norm{\chi_I C_{\varphi_t}} \leq \epsil.
\end{equation}
(Here and throughout the rest of the proof $\norm{\ }$ refers to
the operator norm on $\mathcal{M}$.)

Now fix $t_0 \in [0,1]$ and pick an arc $I_0 \subset \T$ with
$m(I_0) \leq \delta$ whose midpoint is $e^{2\pi it_0}$.  By
Lemma~\ref{le:rho} there exists $\eta > 0$ such that if
$\abs{t_0-t} \leq \eta$, then
$\rho(\varphi_{t_0}(\zeta),\varphi_t(\zeta)) \leq \epsil$ for
all $\zeta \in \T \setminus I_0$. Hence Lemma~\ref{le:MDiff}
shows that
\begin{equation} \label{eq:Est2}
    \norm{\chi_{\T\setminus I_0} (C_{\varphi_{t_0}}-C_{\varphi_t})}
       \leq C\epsil/(1-\abs{\varphi_{t_0}(0)})
\end{equation}
whenever $\abs{t_0-t} \leq \eta$.  To finish the argument we just
write
\[
    C_{\varphi_{t_0}}-C_{\varphi_t}
       = \chi_{I_0}C_{\varphi_{t_0}} - \chi_{I_0}C_{\varphi_t} +
         \chi_{\T\setminus I_0}(C_{\varphi_{t_0}}-C_{\varphi_t})
\]
and, when $\abs{t_0-t} \leq \eta$, invoke estimates \eqref{eq:Est1}
and \eqref{eq:Est2} to conclude that
\[
    \norm{C_{\varphi_{t_0}}-C_{\varphi_t}}
    \leq \epsil + \epsil + C \epsil/(1-\abs{\varphi_{t_0}(0)}).
\]
Since $\epsil > 0$ was arbitrary, this clearly shows that the
norm of $C_{\varphi_{t_0}}-C_{\varphi_t}$ on $\mathcal{M}$ tends to
zero as $t \to t_0$.

The proof of Main Theorem is now complete.

\begin{remark}\label{re:Heur}
We try to describe the heuristics behind the above construction.
First of all, one can easily show that if a continuous path
$(C_{\varphi_t})$ yielding the desired example exists, then one may
assume that the image of each map $\varphi_t$ is contained in the
disc $\{w : \abs{w-\tfrac{1}{2}} \leq \tfrac{1}{2}\}$. Then
$\tau_{1,1}$ is necessarily of the form $g\,dm + \lambda$ where
$g \geq 1$ and $\lambda$ is non-trivial and continuously singular.
One may also assume that $\varphi_0\equiv 0$.
The central issue then is to find the intermediate maps $\varphi_t$
for $0 < t < 1$.  A seemingly natural choice might be
$\varphi_t=(1-t)\varphi_0+t\varphi_1$, but this obviously fails to
work since each such map is compact.  On the other hand, in certain
applications to spectral theory one proceeds by considering the maps
corresponding to the Aleksandrov--Clark measures
$\tau_{t,1}=(1-t)\tau_{0,1}+t\tau_{1,1}$.
However, Theorem~\ref{thm:MacExt} suggests that this approach might
not work either.  Namely, in the case of a discrete
singular part, Theorem~\ref{thm:MacExt} shows that if one makes
a simultaneous change---no matter how small---to all the mass points
of the singular part, then this induces a big difference in the
corresponding composition operator.  These considerations were
behind our actual choice \eqref{eq:tau1def}, where the singularity
$\lambda$ is continuously ``wiped off'' in such a way that the
change in $\tau_{t,1}$ is strictly local at every instant $t$.
\end{remark}

\section{Further remarks and open problems}
\label{sec:Further}

After the work of Section~\ref{sec:Main} it is natural to search for a
larger class of composition operators that could be continuously joined
to the compacts.  For instance, one might be tempted to expect a
positive answer to the following question:
\begin{itemize}
\item
Assume that $\varphi$ and $\alpha_0 \in \T$ are such that the
measure $\tau_{\varphi,\alpha_0}$ has no atoms and, for all
$\alpha \neq \alpha_0$, the measure $\tau_{\varphi,\alpha}$ is
absolutely continuous. Does it follow that $C_\varphi$ belongs to
$\Comp_K(\mathcal{H}^2)$?
\end{itemize}
The answer to this question is, however, negative.

\begin{ejemplo} \label{ex:Isol}
There is a symbol $\psi$ such that $C_\psi$ is isolated in
$\Comp(\mathcal{H}^2)$ and the following properties hold:\
$\tau_{\psi,1}$ has a continuous non-trivial singular part while
all the other measures $\tau_{\psi,\alpha}$ are absolutely
continuous.  In fact, one may choose
$\psi = \varphi \circ \sigma$, where $\sigma$ is
an inner function and $\varphi$ is a conformal map from
$\D$ onto a region $\Omega\subset\D$ with
$\overline{\Omega} \cap \T = \{1\}$.
\end{ejemplo}

The above example is based on a construction of Shapiro and
Sundberg~\cite{ShSuIso}.  We first recall some terminology. Shapiro
and Sundberg call a continuous and $2\pi$-periodic function
$\kappa\colon \mathbb{R} \to \cointerv{0,1}$ a \emph{contact function}
if it is increasing and positive on $\ocinterv{0,\pi}$, decreasing and
positive on $\cointerv{-\pi,0}$ and vanishes at the origin.  Such a
function determines an approach region
\[
    \Omega(\kappa)
    = \{ re^{i\theta}: 0 \leq r < 1-\kappa(\theta) \},
\]
whose boundary is a Jordan curve in $\closure{\D}$ that meets the unit
circle only at the point $1$.  In this setting Shapiro and Sundberg
prove the following (see Theorem~4.1 and Remark~5.1 of \cite{ShSuIso}).

\begin{theorem}[Shapiro--Sundberg] \label{thm:SSIsol}
Suppose $\kappa$ is a $C^2$ contact function and $\varphi$ is a
conformal map from $\D$ onto $\Omega(\kappa)$.  If
\begin{equation} \label{eq:Extremal}
    \int_0^\pi \log\kappa(\theta)\,d\theta = -\infty,
\end{equation}
then $C_\varphi$ is (essentially) isolated in $\Comp(\mathcal{H}^2)$.
\end{theorem}

We observe that this theorem can be extended as follows.

\begin{proposition} \label{prop:IsolExt}
Let $\varphi$ be a function given by Theorem~\ref{thm:SSIsol}, and
let $\sigma$ be an inner function with $\sigma(0) = 0$.  Put
$\psi = \varphi \circ \sigma$.  Then $C_\psi$ is
(essentially) isolated in $\Comp(\mathcal{H}^2)$.
\end{proposition}

Let us note that an analytic self-map of $\D$ is an inner function
if and only if any (or all) of its Aleksandrov--Clark measures is
singular.  Therefore, to produce the symbol needed for
Example~\ref{ex:Isol}, we choose any inner function $\sigma$
vanishing at the origin whose Aleksandrov--Clark measure
$\tau_{\sigma,1}$ is continuously singular.
We then apply Proposition~\ref{prop:IsolExt} with the additional
requirement that $\varphi(1) = 1$.  It is relatively
easy to check that $\psi = \varphi \circ \sigma$ has the required
properties; in particular, $\tau_{\psi,1}$ cannot have atoms.

\begin{proof}[Proof of Proposition~\ref{prop:IsolExt}]
We start by recalling some ideas from the proof of
Theorem~\ref{thm:SSIsol}. Write $\Omega = \Omega(\kappa)$ for the
image of $\varphi$.  A crucial part of Shapiro's and Sundberg's
argument is the construction of a sequence of test functions
$f_n \in \mathcal{H}^2$ which converges to zero weakly in
$\mathcal{H}^2$.  Their functions satisfy the following
properties: $\abs{f_n}^2 \geq c/m(J_n)$ on $\Gamma_n$, where
$\Gamma_n \subset \boundary\Omega$ are arcs converging to $1$ and
$J_n = \varphi^{-1}(\Gamma_n)$; and $\abs{f_n} \leq 1$ on
$\D \setminus T_n$, where $T_n \subset \D$ is a set containing
$\Gamma_n$ whose diameter is roughly twice the length of $\Gamma_n$.
Now suppose that $\eta\colon \D \to \D$ is any analytic map different
from $\varphi$.  Shapiro and Sundberg consider the sets
$E_n = \{ \zeta \in J_n: \abs{\varphi(\zeta)-\eta(\zeta)} \geq c_n\}$
where $c_n$ is approximately twice the diameter of $T_n$.  They
observe that for $\zeta \in E_n$ one has $\varphi(\zeta) \in \Gamma_n$
and $\eta(\zeta) \in \D \setminus T_n$.  Therefore
$\abs{f_n\circ\varphi - f_n\circ\eta}^2 \geq c/m(J_n)$ on $E_n$.
Since $f_n \to 0$ weakly, this yields the estimate
\[
    \norm{C_\varphi-C_\eta}_e^2
       \geq c \limsup_{n\to\infty} \frac{m(E_n)}{m(J_n)}.
\]
Finally Shapiro and Sundberg show that $\limsup m(E_n)/m(J_n) = 1$,
based simply on the fact that
$\int_\T \log\abs{\varphi-\eta} \,dm > -\infty$.

Our argument is just a minor adaptation of the one explained above.
Suppose that $\eta\colon \D \to \D$ is an analytic map different
from $\psi$, and put $J_n' = \psi^{-1}(\Gamma_n)$ and
$E_n' = \{ \zeta \in J_n' : \abs{\psi(\zeta)-\eta(\zeta)} \geq c_n\}$.
Then $J_n' = \sigma^{-1}(J_n)$, and since $\sigma$ is an inner
function fixing the origin, we have $m(J_n') = m(J_n)$.
Thus, using the test functions $f_n$ as before, we arrive at the
estimate
\[
    \norm{C_\psi-C_\eta}_e^2
       \geq c \limsup_{n\to\infty} \frac{m(E_n')}{m(J_n')}.
\]
The proof is now completed by using the same argument as
Shapiro and Sundberg to show that the limit superior here
equals $1$.
\end{proof}

Given the above example, it seems appropriate to close this
section with the following general open problem.

\begin{problem}
Determine all the non-compact composition operators in
$\Comp_K(\mathcal{H}^2)$.
\end{problem}

This problem might be quite hard.  As a first step one could try to
describe interesting subsets of $\Comp_K(\mathcal{H}^2)$ that are
larger than those provided by obvious modifications of our
construction presented in Section~\ref{sec:Main}.  For instance, it
would be instructive to know if the extremality condition
\eqref{eq:Extremal} that was essential for the example provided
by Proposition~\ref{prop:IsolExt} can be relaxed.

\end{document}